\documentclass[11pt,leqno]{article}
\usepackage{amssymb}
\usepackage{graphicx}
\usepackage{amsfonts, amsthm}
\usepackage{amsxtra}
\usepackage{tikz}
\usepackage[latin1]{inputenc}
\usepackage{fontenc}
\usepackage[dvips]{epsfig}
\begin{document}
\newtheorem{thm}{Theorem}[section]
\newtheorem{cor}[thm]{Corollary}
\newtheorem{lem}[thm]{Lemma}
\newtheorem{exm}[thm]{Example}
\newtheorem{prop}[thm]{Proposition}
\theoremstyle{definition}
\newtheorem{defn}[thm]{Definition}
\theoremstyle{remark}
\newtheorem{rem}[thm]{Remark}
\newcommand{\mat}[4]{
    \begin{pmatrix}
           #1 & #2 \\
           #3 & #4
      \end{pmatrix}
   }
\def\Z{\mathbb{Z}} %Integer numbers
\def\R{\mathcal{R}} %Rings obtained by linear combination of matrices
\def\I{\mathcal{I}} %Ideals
\def\C{\mathbb{C}} %Complex numbers
\def\N{\mathbb{N}} %Natural numbers
\def\PP{\mathbb{P}} %Projective space
\def\Q{\mathbb{Q}} %Rational numbers
\def\L{\mathcal{L}} %Transfer operator
\def\ol{\overline} %overline
\def\bs{\backslash} %backslash
\def\part{P} %partition

\newcommand{\gcD}{\mathrm {\ gcd}} %greatest common divisor
\newcommand{\End}{\mathrm {End}} %Endomorphisms
\newcommand{\Aut}{\mathrm {Aut}} %Automorphisms
\newcommand{\GL}{\mathrm {GL}} %General linear group
\newcommand{\SL}{\mathrm {SL}} %Special linear group
\newcommand{\PSL}{\mathrm {PSL}} %Projective linear group
\newcommand{\Mat}{\mathrm {Mat}} %Matrices
\newcommand\ga[1]{\overline{\Gamma}_0(#1)} %The congruence group
\newcommand\pro[1]{\mathbb{P}1(\mathbb{Z}_{#1})} %The arithmetic projective
% line
\newcommand\Zn[1]{\mathbb{Z}_{#1}}
\newcommand\equi[1]{\stackrel{#1}{\equiv}}
\newcommand\pai[2]{[#1:#2]} %the elements of the projective spaces
\newcommand\modulo[2]{[#1]_#2} %calculation of the first number modulo
                              %the second number
\newcommand\sah[1]{\lceil#1\rceil} %The upper integral part of #1
\def\sol{\phi} % a solution of the Lewis equation
\begin{center}
{\LARGE\bf Rigidity for circle diffeomorphisms with  breaks satisfying a Zygmund smoothness condition}
\\
\vspace{.25in} { \large{
H. A. Akhadkulov\footnote{School of Quantitative Sciences, University Utara Malaysia, CAS 06010,
UUM  Sintok, Kedah Darul Aman, Malaysia.\quad E-mail:akhadkulov@yahoo.com},
A. A. Dzhalilov\footnote{Turin Polytechnic University, Kichik Halka yuli 17,  Tashkent 100095, Uzbekistan.  E-mail: a\_dzhalilov@yahoo.com} and
K. M. Khanin\footnote{Department of Mathematics, University of Toronto, 40 St. George Street, Toronto, Ontario M5S 2E4, Canada.
E-mail: khanin@math.toronto.edu}
}}

\end{center}

\title{Rigidity for circle diffeomorphisms with  a break satisfying a Zygmund smoothness condition}

\begin{abstract}
Let $f$ and $\tilde{f}$ be two circle diffeomorphisms
with a break point,  with the same irrational rotation number of bounded type,
the same size of the break $c$ and satisfying   a certain Zygmund type smoothness condition
depending on a parameter $\gamma>2.$ We prove that under a certain condition imposed on
the break size $c$, the diffeomorphisms
$f$ and $\tilde{f}$ are $C^{1+\omega_{\gamma}}$-smoothly conjugate to each other,
where $\omega_{\gamma}(\delta)=|\log \delta|^{-(\gamma/2-1)}.$

\end{abstract}
\section{Introduction}
The problem of smoothness of a conjugacy between two circle  diffeomorphisms  is a classical problem in one-dimensional dynamics.
 Arnol'd \cite{Ar1961}  proved that any analytic circle diffeomorphism with a Diophantine rotation number, sufficiently close to the rigid rotation $f_{\rho}\rightarrow x+\rho$ is analytically conjugate to $f_{\rho}.$
First significant extension of Arnol'd's result was obtained by Herman \cite{He1979}. He proved that 
$C^{\infty}$-smooth circle diffeomorphism with a Diophantine rotation number is $C^{\infty}$-conjugate to $f_{\rho}.$
 Last forty years Herman's result was developed by Yoccoz \cite{Yo1984},  Khanin and Sinai \cite{KS1989}, Katznelson and Ornstein \cite{KO1989.1, KO1989.2}, and Khanin and Teplinsky \cite{KT2009} in virtue of their great discoveries, new ideas, methods, and phenomena. 
Summarising thus far,  if $f$ is $C^{2+\nu}$  and the rotation number  satisfies a certain Diophantine condition, then the conjugacy is $C^{1+\alpha}$
for some $0<\alpha<\nu.$ Moreover, in \cite{KO1989.2}, the authors considered a class of circle diffeomorphismsm bigger than $C^{2+\nu}.$
They proved that if $Df$ absolutely continuous and 
$D\log Df\in L_{p},$ for some $p>1$ then the conjugacy is   
is absolutely continuous provided its rotation number is bounded type.
 One of the last results on the progression of the regularity of conjugacy of circle diffeomorphisms have been contributed by
 Akhadkulov \emph{et al} \cite{{ADK2017}} by extending previous results for circle diffeomorphisms satisfying a certain Zygmund-type smoothness condition
 depending on a parameter $\gamma>0.$ 
It was shown that, if a circle diffeomorphism satisfies the  Zygmund condition for $\gamma>1/2$ then there exists a subset of irrational numbers of unbounded type such that the conjugacy is absolutely continuous provided its rotation number belongs to the above set.       
 Moreover, if  $\gamma>1$ then  the conjugacy is $C^{1}$-smooth for almost all irrational rotation numbers.
 It is important to remark that, in the case of diffeomorphisms, rigidity is guaranteed only
 when the rotation numbers  satisfy a certain Diophantine condition. 
 Recently, Khanin and Teplinsky \cite{KT2007} showed that in the presence  of \emph{critical points} or 
  \emph{break points}  points the rigidity may be stronger, i.e., valid for a "large" set of rotation numbers.  They have  showed that for
  the  diffeomorphisms of a circle with a single critical point, the \emph{robust rigidity} holds,
that is, the rigidity holds without any  Diophantine conditions.
The robust rigidity result depends  on exponential convergence of renormalizations so called \textit{renormalization} problem. The  
renormalization problem was proved by de Faria and de Melo for $C^{\infty}$-smooth critical
circle maps with irrational rotation numbers of bounded type \cite{MeloFarie I, MeloFarie II},
 and extended, in the analytic setting,  by Yampolsky \cite{Yampol} to cover
all irrational rotation numbers.
Recently, a remarkable rigidity  results also have been obtained by Guarino and de Melo \cite{pablo1} and Guarino \emph{ et al} \cite{pablo2}
in the case of  lower smoothness of critical circle maps. In \cite{pablo1}, it was proven a $C^{1+\alpha}$ (for a universal $\alpha>0$) rigidity  result for any two $C^3$
critical circle maps with the same irrational rotation number of bounded type and the same odd criticality. 
In the case of the class is  $C^{4},$ $C^{1}$-rigidity holds for any irrational rotation number and  $C^{1+\alpha}$- rigidity holds  for a full Lebesgue measure set of rotation numbers as shown in  \cite{pablo2}.

In the case of a break type singularity, the first rigidity results for $C^{2+\alpha}$ circle diffeomorphisms were obtained by Khanin and
Khmelev  \cite{KhKhem2003}, and  Khanin and Teplinsky \cite{KT2013}. 
In \cite{KhKhem2003}, rigidity theorem was proved 
for irrational rotation numbers with periodic \textit{partial quotients} and in \cite{KT2013}, for \textit{half bounded} (see the definition below) irrational rotation numbers.
Note that the robust rigidity does not hold for circle diffeomorphisms with breaks. Indeed, as shown in \cite{KhKocic2013}, there are irrational rotation numbers,
and pairs of analytic circle diffeomorphisms with breaks, with the same rotation number
and the same size of the break, for which any conjugacy between them is
not even Lipschitz continuous. The most remarkable results  in this direction
were obtained by Khanin and Koci\`{c} \cite{KSasa} and Khanin \emph{et al} \cite{KSE}.
In \cite{KSasa}, it was shown that the renormalizations of any two $C^{2+\alpha}$-smooth  circle diffeomorphisms
with a break point, with the same irrational rotation number and the same
size of the break, approach each other exponentially fast in the $C^{2}$-topology. 
This result implies that 
for almost
all irrational numbers, any two $C^{2+\alpha}$-smooth circle diffeomorphisms with a break, with the same
rotation number and the same size of the break, are $C^{1}$-smoothly conjugate to each
other as shown in \cite{KSE}.
The interesting problems  of circle maps are  the  rigidity and renormalizations problems
on the less regularities, for instance these  problems are  open for $C^{2+\alpha}$-smooth critical circle maps and
for circle diffeomorphisms with break points satisfying a Zygmund condition, even for bounded combinatorics.
The  renormalizations problem for circle diffeomorphisms
with a break satisfying a certain Zygmund condition is partially
solved in \cite{Renorm}.

In this paper we study the rigidity problem of
 two circle diffeomorphisms $f$ and $\tilde{f}$ with a break point,
 with the same irrational rotation number of bounded type,
the same size of the break $c$ and satisfying   a certain Zygmund type smoothness condition
depending on a parameter $\gamma>2.$ We prove that under a certain condition imposed on
the break size $c,$ the diffeomorphisms $f$
and $\tilde{f}$ are $C^{1+\omega_{\gamma}}$-smoothly conjugate to each other,
where $\omega_{\gamma}(\delta)=|\log \delta|^{-(\gamma/2-1)}.$
 The rest of this paper is organized as follows. In Section 2,
 the main notions and statement of main theorem are given.
In Section 3,  we show the existence of a solution of a cohomological equation
for the break-equivalent diffeomorphisms.
 In Section 4, some universal estimates for the ratio
 of the lengths of the segments of dynamical partition are obtained.
 Sections 5 and 6 are devoted to study the renormalizations and  closeness of rescaled points.
 Finally, in Section 7, the proof of main theorem is given.

\section{General settings and statement of main Theorem}
\subsection{Dynamical partition}
In this section, first we present some of the basic notations of circle maps and then  we estimate the ratio of lengths  of elements of the dynamical partition. Denote by  $\mathbb{S}^{1}=\mathbb{R}/\mathbb{Z}$  unit circle. 
Let  $f:\mathbb{S}^{1}\rightarrow \mathbb{S}^{1}$
be a circle homeomorphism  we denote its rotation number by
$\rho(f).$ It can be expressed as a continued fraction
$$
\rho(f)=1/(k_{1}+1/(k_{2}+...)):=[k_{1}, k_{2},...,k_{n},...).
$$
The sequence of positive integers $(k_{n})$ with $n\geq 1$  called
 \emph{partial quotients} and it is infinite if and only if $\rho(f)$ is irrational.
 We call $\rho:=\rho(f)$ is bounded type if $s(\rho):=\sup k_{n}<\infty.$
Let $p_{n}/q_{n}=[k_{1}, k_{2},...,k_{n}]$ 
be the sequence of rational convergents
of  $\rho.$
 The coprime numbers $p_{n}$ and $q_{n}$
satisfy the recurrence relations 
$$
p_{n}=k_{n}p_{n-1}+p_{n-2},\,\,\text{and
}\,\,q_{n}=k_{n}q_{n-1}+q_{n-2}
$$ 
 for $n\geq 1,$ where$p_0=0,$ $q_0=1$ and $p_{-1}=1,$ $q_{-1}=0.$
Let $\xi_{0}\in \mathbb{S}^{1}.$  Define
 $n$th \emph{fundamental segment} $\Delta^{(n)}_{0}:=\Delta^{(n)}_{0}(\xi_{0})$ as the circle arc
  $[\xi_{0}, f^{q_{n}}(\xi_{0})]$ if $n$ is even and $[f^{q_{n}}(\xi_{0}), \xi_{0}]$ if $n$ is odd.
  We shall also use the notations $\widehat{\Delta}^{(n-1)}_{0}=\Delta^{(n)}_{0} \cup \Delta^{(n-1)}_{0}$
  and $\check{\Delta}^{(n-1)}_{0}=\Delta^{(n-1)}_{0} \setminus \Delta^{(n+1)}_{0}.$
Certain number of images of fundamental segments $\Delta^{(n-1)}_{0}$ and $\Delta^{(n)}_{0},$ under the iterates of
 $f,$ cover whole circle without overlapping beyond the endpoints and form
$n$th \emph{dynamical partition} of the circle $\mathbb{S}^{1}$
$$
\mathcal{P}_{n}:=\mathcal{P}_{n}(\xi_{0}, f)=\left\lbrace \Delta_{j}^{(n)}:=f^{j}(\Delta^{(n)}_{0}),
 0\leq j<q_{n-1}\right\rbrace
\bigcup \left\lbrace \Delta_{i}^{(n-1)}:=f^{i}(\Delta^{(n-1)}_{0}),  0\leq i<q_{n}\right\rbrace.
$$
The partition $\mathcal{P}_{n+1}$  is a refinement of the partition $\mathcal{P}_{n}.$ Indeed,
the segments of order $n$ belong to $\mathcal{P}_{n+1}$ and each segment
$\Delta_{i}^{(n-1)},$ $0\leq i <q_{n}$ is partitioned into
$k_{n+1}+1$ segments belonging to $\mathcal{P}_{n}$ such that
\begin{equation}\label{partition}
\Delta_{i}^{(n-1)}=\Delta_{i}^{(n+1)}\cup\bigcup_{s=0}^{k_{n+1}-1}\Delta_{i+q_{n-1}+sq_{n}}^{(n)}.
\end{equation}
One can easily see that the endpoints of the segments from $\mathcal{P}_{n}$ form the set
$$
\Xi_{n}=\{\xi_{i}:=f^{i}(\xi_{0}),\, 0\leq i<q_{n}+q_{n-1}\}.
$$
We shall also use the extended set $\Xi^{\ast}_{n}=\Xi_{n}\cup \{\xi_{q_{n}+q_{n-1}}\}.$
Now we formulate a lemma which will be used in the sequel.
\begin{lem}\label{lemma1}
For every $m>n,$ we have the following decomposition
\begin{equation}\label{J1}
\Xi_{m}\cap \check{\Delta}^{(n-1)}_{0}=\bigcup_{\xi_{l}\in \Xi_{m}\cap\Delta^{(n)}_{0}\setminus \{\xi_{q_{n}}\}}
\bigcup_{s=0}^{k_{n+1}-1}\xi_{l+sq_{n}+q_{n-1}}.
\end{equation}
Furthermore, for every $\xi_{l}\in \Xi_{m}\cap\Delta^{(n)}_{0}\setminus \{\xi_{q_{n}}\}$
we have $\xi_{l+k_{n+1}q_{n}+q_{n-1}}=\xi_{l+q_{n+1}}\in \Xi^{\ast}_{m}\cap \widehat{\Delta}^{(n)}_{0}.$
\end{lem}
\begin{proof}
The proof of the lemma follows directly from the properties
of dynamical partition.
\end{proof}
%%%%%%%%%%%%%%%%%%%%%%%%%%%%%%%%%%%%%%%%%%%%
%%%%%%%%%%%%%%%%%%%%%%%%%%%%%%%%%%%%%%%%%%%
\subsection{Circle diffeomorphisms with a break and Zygmund class}
We recall the following definition.
\begin{defn}
$f:\mathbb{S}^{1}\rightarrow \mathbb{S}^{1}$ is called  a circle diffeomorphism  with a single break point
 $\xi_{0}$ if the following
conditions are satisfied:
 \begin{itemize}
   \item [(i)] $f\in C^{1}([\xi_{0}, \xi_{0}+1]);$
   \item [(ii)] $\inf_{\xi\neq \xi_{0}} Df(\xi)>0;$
   \item [(iii)] $f$ has  one-sided derivatives $Df(\xi_{0}\pm 0)>0$ and
   $$
   c:=c_{f}(\xi_{0})=\sqrt{\frac{Df(\xi_{0}-0)}{Df(\xi_{0}+0)}}\neq 1.
   $$
 \end{itemize}
\end{defn}
The number $c$ is called the \emph{size of break} of $f$ at  $\xi_{0}.$
Circle diffeomorphisms with a break  were first studied by Khanin \& Vul
in \cite{KV1991}. It was proven  that the renormalizations
circle diffeomorphisms with a break approximate fractional linear transformations.
Next we define a  class of circle diffeomorphisms
 with breaks satisfying a Zygmund condition.
  Consider the function  $\mathcal{Z}_{\gamma}:[0,1)\rightarrow [0, +\infty)$ defined
  as
  $$
  \mathcal{Z}_{\gamma}(x)=|\log x|^{-\gamma}, \,\,\,\,\text{for}\,\,\,x\in (0,1)
  $$
and $\mathcal{Z}_{\gamma}(0)=0,$ where $\gamma>0.$
 Let $f$ be a circle diffeomorphism
with the break point $\xi_{0}.$  Denote by $\nabla^{2}f(\xi, \tau)$
the \emph{second symmetric difference} of $Df,$ that is
$$
\nabla^{2}f(\xi, \tau)=Df(\xi+\tau)+Df(\xi-\tau)-2Df(\xi)
$$
where  $\xi \in \mathbb{S}^{1}\setminus \{\xi_{0}\}$ and  $\tau\in [0,\frac{1}{2}].$
Suppose that there exists a constant $C>0$ such that
\begin{equation}\label{ok1}
 \|\nabla^{2}f(\cdot, \tau)\|_{L^{\infty}(\mathbb{S}^{1})}\leq C\tau\mathcal{Z}_{\gamma}(\tau).
\end{equation}
%The class of real  functions satisfying (\ref{ok1})
%with $\mathcal{Z}_{\gamma}(\tau)$ replaced by  1,  is called the Zygmund class
%(see \cite{Zyg}, p.43).
 %The Zygmund class was applied  to the theory of circle homeomorphisms
%for the first time by Jun Hu \& Sullivan  \cite{JSull1997}, \cite{Sull1992}.
In this work we study the class of circle diffeomerphisms $f$  with  break point $\xi_{0},$
 whose  derivatives $Df$ have bounded variation and
satisfy  the  inequality (\ref{ok1}). We denote this class by
$\mathrm{D}^{1+\mathcal{Z}_{\gamma}}(\mathbb{S}^{1}\setminus \{\xi_{0}\}).$
\begin{rem}
Note that the class $\mathrm{D}^{1+\mathcal{Z}_{\gamma}}(\mathbb{S}^{1}\setminus \{\xi_{0}\})$
is bigger than $C^{2+\epsilon}(\mathbb{S}^{1}\setminus \{\xi_{0}\})$ for any positive
 $\gamma$ and $\epsilon.$
\end{rem}

%and
\subsection{Statement of the main theorem}
In this section we formulate our main theorem.
For this, let us first define some necessary facts.
Let $m\in \mathbb{N}.$
Define
$$
\mathfrak{D}^{(1)}_{m}=\{c\in \mathbb{R}_{+}\setminus \{1\}: \,\,\, c^{4m}-c^{2}<1\};
\,\,\,\,\,\,\,
\mathfrak{D}^{(2)}_{m}=\{c\in \mathbb{R}_{+}\setminus \{1\}: \,\,\, c^{4m+2}+c^{4m}>1\}.
$$
The following is our main theorem.

\begin{thm}\label{Main}
Let $\gamma> 2$ and $m\in \mathbb{N}.$
Let $f$ and $\tilde{f}$ be two circle diffeomorphisms
with a break   satisfying the following conditions:
 \begin{itemize}
 \item[(a)] $f, \tilde{f} \in \mathrm{D}^{1+\mathcal{Z}_{\gamma}}(\mathbb{S}^{1}\setminus \{\xi_{0}\});$
 \item[(b)] $f$ and $\tilde{f}$ have the same irrational rotation number $\rho$
 of bounded type  such that $s(\rho)=m;$
 \item[(c)] $f$ and $\tilde{f}$ have the same size of the break $c\in \mathbb{R}_{+}\setminus \{1\};$
 \item[(d)] $c\in \mathfrak{D}^{(1)}_{m}$ in case of $c>1$ or
 $c\in \mathfrak{D}^{(2)}_{m}$ in case of $0<c<1.$
 \end{itemize}
  Then there exists a $C^{1}$-smooth circle diffeomorphism $h$  and a constant $A>0$
 such that $h\circ f= \tilde{f}\circ h$ and
$$
|Dh(x)-Dh(y)|\leq A \omega_{\gamma}(|x-y|)
$$
for any $x, y \in \mathbb{S}^{1}$ such that  $x\neq y.$
  \end{thm}
  \begin{rem}
The reason for the restriction $c$ in condition $(d)$ is purely technical.
 It enables us to get an algebraic estimate for the ratio of lengths of segments
 $\Delta^{n+\ell}$ and $\Delta^{n}$ satisfying $\Delta^{n+\ell}\subset \Delta^{n}$ of the dynamical partition $\mathcal{P}_{n}$  while $\ell$ has a form of the logarithm of $n.$
We do not know if the statement of Theorem \ref{Main} holds  when the restriction
is removed.

  \end{rem}
%%%%%%%%%%%%%%%%%%%%%%%%%%%%%%%%%%
%%%%%%%%%%%%%%%%%%%%%%%%%%%%%%%%%%%%%%%%%%%%%

%\begin{figure}[h]
%    \centering
%    \includegraphics[scale=0.7]{Habib}\\
%    \caption{\small{The set $\mathfrak{D}_{m}$ is the  interval $(v_{m}^{1}, v_{m}^{2}).$}}
%%    \label{fig:mesh1}
%\end{figure}

%%%%%%%%%%%%%%%%%%%%%%%%%%%%%%%%%%%%%%%%%%

%%%%%%%%%%%%%%%%%%%%%%%%%%%%%%%%%%%
%%%%%%%%%%%%%%%%%%%%%%%%%%%%%%%%%%%%%%%%%%%%%%%%
\section{Cohomological equation for the break-equivalent diffeomorphisms}
In this section we show the existence of a solution of a cohomological equation
for the break-equivalent diffeomorphisms.
We begin from the following definition.
\begin{defn}\label{break-equ}
We say that two circle diffeomorphisms  $f$ and $\tilde{f}$  with a
break $\xi_{0}$ are break-equivalents if
there exists a topological conjugacy $h$
such that $h(\xi_{0})=\xi_{0}$ and $c_{f}(\xi_{0})=c_{\tilde{f}}(h((\xi_{0}))).$
\end{defn}
Consider  two break-equivalent circle diffeomorphisms
$f$ and $\tilde{f}$  with irrational rotation number.
Let $h:\mathbb{S}^{1}\rightarrow \mathbb{S}^{1}$ be the conjugacy between $f$ and $\tilde{f},$ that is,
\begin{equation}\label{F0}
h\circ f= \tilde{f}\circ h.
\end{equation}
The \emph{cohomological equation}
associated to (\ref{F0}) is
\begin{equation}\label{F0a}
\zeta \circ f-\zeta=\log D\tilde{f}\circ h-\log Df
\end{equation}
where $\zeta: \mathbb{S}^{1}\rightarrow \mathbb{R}$
is called the solution of (\ref{F0a}) if it exists.
Note that here $D\tilde{f}(h(x))$ means the derivative of
$\tilde{f}$ at $h(x).$   Define
$$
\Lambda_{n}(x)=\log Df^{q_{n}}(x)-\log D\tilde{f}^{q_{n}}(h(x)),
 \,\,\,\,\, x\in \widehat{\Delta}^{(n-1)}_{0}.
$$
 Since $f$ and $\tilde{f}$ are break-equivalents
one-side limits of $\Lambda_{n}$ at the break point $\xi_{0}$ are equal that is,
$\Lambda_{n}(\xi_{0}-0)=\Lambda_{n}(\xi_{0}+0).$
Therefore $\Lambda_{n}$ is continuous on  $\widehat{\Delta}^{(n-1)}_{0}$
and it can be decomposed as
$$
\Lambda_{n}(x)=\sum_{s=0}^{q_{n}-1}\log Df(f^{s}(x))-\log D\tilde{f}(h\circ f^{s}(x)),
\,\,\,\,\,\,\,\, x\in \widehat{\Delta}^{(n-1)}_{0}.
$$
Denote $\Lambda_{n}=\max_{x\in \widehat{\Delta}^{(n-1)}_{0}}|\Lambda_{n}(x)|.$
The following theorem will be used in the proof of main theorem.
\begin{thm}\label{criteria}
Let $f$ and $\tilde{f}$ be two break-equivalent circle diffeomorphisms with a break
and  with identical irrational rotation number
$\rho=[k_{1},k_{2},...,k_{n},...].$
If
$$
\sum_{n=0}^{\infty}k_{n+1}\Lambda_{n}<\infty
$$
then the cohomological equation (\ref{F0a}) has a continuous solution.
\end{thm}
\begin{proof}
Let $i_{n}:\mathbb{S}^{1}\rightarrow \mathbb{N}_{0}$ be the first entrance time of $x$
in $\widehat{\Delta}^{(n-1)}_{0};$ that is,
$$
i_{n}(x)=\min \{i\geq 0:\, f^{i}(x)\in \widehat{\Delta}^{(n-1)}_{0}\}.
$$
Define $\zeta_{n}:\mathbb{S}^{1}\rightarrow \mathbb{R}$ as follows
$$
\zeta_{n}(x)=\sum_{s=0}^{i_{n}(x)-1}\log Df(f^{s}(x))-\log D\tilde{f}(h\circ f^{s}(x)).
$$
 Next we show that $\zeta_{n}$ is a Cauchy.
 For this,  first we estimate
$\|\zeta_{n+1}-\zeta_{n}\|_{\infty}.$
To estimate this we distinguish the following three cases: \\
\emph{Case I.} Suppose $x\in \mathbb{S}^{1}\setminus \Xi_{n+1}.$
By the definition of $i_{n}$ we have
\[i_{n}(x)=\left\{\begin{array}{ll}
0, & \mbox{if \,\,$x\in \widehat{\Delta}^{(n-1)}_{0}$}\\
q_{n-1}-j, & \mbox{if \,\,$x\in \Delta^{(n)}_{j} $}\\
q_{n}-i, & \mbox{if \,\,$x\in \Delta^{(n-1)}_{i}$}
\end{array}\right.\]
where $0<j<q_{n-1}$  and $0<i<q_{n}.$ Using the properties of dynamical
partition we can show that
\[i_{n+1}(x)-i_{n}(x)=\left\{\begin{array}{ll}
0, & \mbox{if \,\,$x\in \widehat{\Delta}^{(n)}_{0}\cup \Delta^{(n+1)}_{i}$}\\
k_{n+1}q_{n}, & \mbox{if \,\,$x\in \Delta^{(n)}_{j}$}\\
(k_{n+1}-\ell-1)q_{n}, & \mbox{if \,\,$x\in \Delta^{(n)}_{i+q_{n-1}+\ell q_{n}}$}
\end{array}\right.\]
where $0<j<q_{n-1},$ $0<i<q_{n}$ and $0\leq \ell <k_{n+1}.$ Therefore
$|\zeta_{n+1}(x)-\zeta_{n}(x)|=0$ if  $x\in \widehat{\Delta}^{(n)}_{0}\cup \Delta^{(n+1)}_{i},$
$0<i<q_{n}$ and
\begin{equation}\label{F1}
\begin{split}
  |\zeta_{n+1}(x)-\zeta_{n}(x)| &= \Big|\sum_{s=i_{n}(x)}^{i_{n+1}(x)-1}
  \log Df(f^{s}(x))-\log D\tilde{f}(h\circ f^{s}(x))\Big| \\
  &=\Big|\sum_{s=0}^{i_{n+1}(x)-i_{n}(x)-1}
  \log Df(f^{s}(x_{i_{n}})-\log D\tilde{f}(h\circ f^{s}(x_{i_{n}}))\Big| \\
  &\leq \Big|\sum_{s=0}^{q_{n}-1}
  \log Df(f^{s}(x_{i_{n}})-\log D\tilde{f}(h\circ f^{s}(x_{i_{n}}))\Big| \\
  &+\Big|\sum_{s=q_{n}}^{2q_{n}-1}
  \log Df(f^{s}(x_{i_{n}})-\log D\tilde{f}(h\circ f^{s}(x_{i_{n}}))\Big| \\
  &\vdots\\
  &+\Big|\sum_{s=i_{n+1}(x)-i_{n}(x)-q_{n}}^{i_{n+1}(x)-i_{n}(x)-1}
  \log Df(f^{s}(x_{i_{n}})-\log D\tilde{f}(h\circ f^{s}(x_{i_{n}}))\Big|
   \end{split}
\end{equation}
if $x\in \Delta^{(n)}_{j},$ $0<j<q_{n+1}$ where $x_{i_{n}}=f^{i_{n}(x)}(x).$
Clearly $f^{i_{n}}$ maps $\mathbb{S}^{1}$ into $\widehat{\Delta}^{(n-1)}_{0}$
 and the points  $x_{i_{n}}, f^{q_{n}}(x_{i_{n}}),..., f^{i_{n+1}(x)-i_{n}(x)-q_{n}}(x_{i_{n}})$
lie in the interval $\widehat{\Delta}^{(n-1)}_{0}.$ Therefore the right hand side of
(\ref{F1}) can be estimated as follows
$$
|\zeta_{n+1}(x)-\zeta_{n}(x)|\leq \sum_{s=0}^{i_{n+1}(x)-i_{n}(x)-q_{n}}\Big|\Lambda_{n}\Big(f^{sq_{n}}(x_{i_{n}})\Big)\Big|
\leq k_{n+1}\Lambda_{n}.
$$
Hence
\begin{equation}\label{F2}
\|\zeta_{n+1}-\zeta_{n}\|_{\infty}\leq k_{n+1}\Lambda_{n}.
\end{equation}
\emph{Case II.} Suppose $x=\xi_{i}\in \Xi_{n}.$
For $i=0,$ it is clear that $|\zeta_{n+1}(\xi_{0})-\zeta_{n}(\xi_{0})|=0.$
For $i\geq 1,$ one can easily see
\[i_{n}(\xi_{i})=\left\{\begin{array}{ll}
q_{n-1}-i, & \mbox{if \,\,$1\leq i \leq q_{n-1}$}\\
q_{n}-i, & \mbox{if \,\,$q_{n-1}< i \leq q_{n}$}\\
q_{n}+q_{n-1}-i, & \mbox{if \,\,$q_{n}<i<q_{n}+q_{n-1}.$}
\end{array}\right.\]
Consequently, we get
\[i_{n+1}(\xi_{i})=\left\{\begin{array}{ll}
q_{n}-i, & \mbox{if \,\,$1\leq i \leq q_{n}$}\\
q_{n+1}-i, & \mbox{if \,\,$q_{n}< i < q_{n}+q_{n-1}.$}
\end{array}\right.\]
Therefore
\[i_{n+1}(\xi_{i})-i_{n}(\xi_{i})=\left\{\begin{array}{ll}
q_{n}-q_{n-1}, & \mbox{if \,\,$1\leq i \leq q_{n-1}$}\\
0, & \mbox{if \,\,$q_{n-1}< i \leq q_{n}$}\\
(k_{n+1}-1)q_{n}, & \mbox{if \,\,$q_{n}<i<q_{n}+q_{n-1}.$}
\end{array}\right.\]
This and by the definition of $\zeta_{n}$ we have
$|\zeta_{n+1}(\xi_{i})-\zeta_{n}(\xi_{i})|=0$ if $q_{n-1}< i \leq q_{n},$
and
\begin{equation}\label{F3}
|\zeta_{n+1}(\xi_{i})-\zeta_{n}(\xi_{i})|=
|\Lambda_{n}(\xi_{0})-\Lambda_{n-1}(\xi_{0})|\leq
\Lambda_{n}+\Lambda_{n-1}
\end{equation}
if $1\leq i \leq q_{n-1}$ and
\begin{equation}\label{F4}
|\zeta_{n+1}(\xi_{i})-\zeta_{n}(\xi_{i})|=
\Big|\sum_{s=1}^{k_{n+1}-1}\Lambda_{n}(\xi_{s q_{n}+q_{n-1}})\Big|\leq
k_{n+1}\Lambda_{n}.
\end{equation}
if $q_{n}<i <q_{n}+ q_{n-1}.$\\
\\
\emph{Case III.} Suppose  $x=\xi_{i}\in \Xi_{n+1}\setminus \Xi_{n}.$
In this case we consider the following sub-cases:
\begin{align*}
a)\,\,\, i\in & L_{n}:=\{\ell q_{n}+q_{n-1}, \, 1\leq \ell <k_{n+1}\}, &
b)\,\,\, i\in &(q_{n}+q_{n-1}, q_{n+1})\setminus L_{n},\\
c)\,\,\,i= & q_{n+1},  & d)\,\,\,i\in &(q_{n+1}, q_{n+1}+q_{n}).
\end{align*}
It is easy to check that $i_{n}(\xi_{i})=0$ and
$i_{n+1}(\xi_{i})=(k_{n+1}-\ell)q_{n}$
in the sub-case of $a).$
Thus one gets
\begin{equation}\label{F5}
|\zeta_{n+1}(\xi_{i})-\zeta_{n}(\xi_{i})|=
\Big|\sum_{s=\ell}^{k_{n+1}-1}\Lambda_{n}(\xi_{s q_{n}+q_{n-1}})\Big|\leq
k_{n+1}\Lambda_{n}.
\end{equation}
Consider the sub-case  $b).$
It is clear that $i$ can be written as
$i=\ell_{1}q_{n}+q_{n-1}+i_{1}$ for some $1\leq \ell_{1} <k_{n+1}$ and $1\leq i_{1} <q_{n}.$
By the definition of $i_{n}$ we have
$i_{n}(\xi_{i})=q_{n}-i_{1}$ and $i_{n+1}(\xi_{i})=q_{n+1}-i=
(k_{n+1}-\ell_{1})q_{n}-i_{1}.$
It implies
\begin{equation}\label{F6}
|\zeta_{n+1}(\xi_{i})-\zeta_{n}(\xi_{i})|=
\Big|\sum_{s=\ell_{1}}^{k_{n+1}-1}\Lambda_{n}(\xi_{s q_{n}+q_{n-1}})\Big|\leq
k_{n+1}\Lambda_{n}.
\end{equation}
The sub-case $c)$ is clear because of both functions
$\Lambda_{n}$ and $\Lambda_{n+1}$ are zero at  $\xi_{i}.$
Finally, consider the sub-case $d).$ In this case
$i$ can be written as $i=q_{n+1}+i_{1}$ for some $1\leq i_{1}< q_{n}.$
One can easily see $i_{n}(\xi_{i})=q_{n}-i_{1}$ and $i_{n+1}(\xi_{i})=q_{n+1}+q_{n}-i
=q_{n+1}+q_{n}-(q_{n+1}+i_{1})=q_{n}-i_{1}$
which  implies
\begin{equation}\label{F7}
|\zeta_{n+1}(\xi_{i})-\zeta_{n}(\xi_{i})|=0.
\end{equation}
 Combining the inequalities (\ref{F2})-(\ref{F7}) we obtain, finally,
 \begin{equation}\label{F7a}
 \|\zeta_{n+1}-\zeta_{n}\|_{\infty}\leq k_{n}\Lambda_{n-1}+
 k_{n+1}\Lambda_{n}.
 \end{equation}
 From this it follows that
  \begin{equation}\label{F7b}
 \|\zeta_{n+p}-\zeta_{n}\|_{\infty}\leq 2\sum_{m=n}^{n+p} k_{m}\Lambda_{m-1}.
  \end{equation}
  Thus $\zeta_{n}$ is a Cauchy.
  Let $\zeta(x)=\lim_{n\rightarrow \infty}\zeta_{n}(x).$
 Next we show that
 the  function  $\zeta:\mathbb{S}^{1}\rightarrow \mathbb{R}$ is continuous and satisfies
the cohomological equation (\ref{F0a}).
First we show that $\zeta$ satisfies (\ref{F0a}).
It is easy to see that for any $x\in \mathbb{S}^{1}\setminus \{\xi_{0}\}$ there exists
$n_{0}:=n_{0}(x)$ such that $i_{n}(f(x))=i_{n}(x)-1$ for all $n\geq n_{0}.$
This and by the definition of $\zeta_{n}$ we get
$$
\zeta_{n}\circ f-\zeta_{n}=\log D\tilde{f}\circ h-\log Df
$$
for all $n\geq n_{0}.$ Taking the limit as $n\rightarrow \infty$
we get (\ref{F0a}). Let $x=\xi_{0}.$ It is easy to see that
$\zeta_{n}(\xi_{0})=0$ and
\begin{equation} \label{F9}
\begin{split}
\zeta_{n}(f(\xi_{0})) & =
\sum_{s=0}^{i_{n}(f(\xi_{0}))-1}\log Df(f^{s+1}(\xi_{0}))-\log D\tilde{f}(h\circ f^{s+1}(\xi_{0})) \\
 & = \sum_{s=0}^{q_{n-1}-2}\log Df(f^{s+1}(\xi_{0}))-\log D\tilde{f}(h\circ f^{s+1}(\xi_{0}))\\
 &=\Lambda_{n-1}(\xi_{0})+\log D\tilde{f}(h(\xi_{0}))-\log Df(\xi_{0}).
\end{split}
\end{equation}
Taking the limit as $n\rightarrow \infty$ we again get (\ref{F0a}).
Next we show that $\zeta$ is continuous at $x=\xi_{0}.$
Since $\zeta_{n}(\xi_{0})=0$ for all $n \geq 1$ we have
$\zeta(\xi_{0})=0.$ Take any $z\in \widehat{\Delta}^{(n-1)}_{0}.$
It is obvious that  $i_{j}(z)=0$ for every $j\leq n,$ so $\zeta_{j}(z)=0$
for every  $j\leq n.$ In particular
$$
\zeta_{n+p}(z)=\sum_{m=0}^{p-1} \zeta_{n+m+1}(z)-\zeta_{n+m}(z).
$$
This and relation (\ref{F7a})  imply
$$
|\zeta_{n+p}(z)|\leq 2\sum_{m=n}^{n+p} k_{m}\Lambda_{m-1}.
$$
Consequently
$$
\lim_{n\rightarrow \infty}\sup_{z\in \widehat{\Delta}^{(n-1)}_{0}}|\zeta(z)|=0.
$$
Hence $\zeta$ is continuous at $x=\xi_{0}.$ Denote by $\Xi=\{\xi_{i}:=f^{i}(\xi_{0}),\,i\in\mathbb{N}\}$
the positive trajectory of $\xi_{0}.$ Since $\zeta$ is continuous at $x=\xi_{0}$ and
$\log D\tilde{f}\circ h-\log Df$
is continuous on $\mathbb{S}^{1},$ by
$$
\zeta \circ f-\zeta=\log D\tilde{f}\circ h-\log Df
$$
it  implies that   $\zeta$ is continuous
on $\Xi.$ Note that $i_{n}:\mathbb{S}^{1}\rightarrow \mathbb{R}$ is continuous in the interior
of each element of the partition $\mathcal{P}_{n}$ for every $n\geq 1.$
As a consequence $\zeta_{n}$ is continuous  in the interior
of each element of the partition $\mathcal{P}_{n}$ for every $n\geq 1.$
Thus the limit function $\zeta$ is continuous on $x\in \mathbb{S}^{1}\setminus \Xi.$
\end{proof}
\begin{rem}
It is important to remark that Theorem \ref{criteria} holds true for
any two break-equivalent circle diffeomorphisms with any countable number of break points.
\end{rem}
\section{Renormalizations of circle diffeomorphisms  with a break}
In this section we will discuss  on convergence of renormalizations
of two circle diffeomorphisms  with a break.
%satisfying the inequality
%(\ref{ok1}).
Let us recall first the definition of renormalization of  circle maps.
The  segment $\widehat{\Delta}^{(n-1)}_{0}$
is called the $n^{\text{th}}$ \emph{renormalization neighborhood} of  $\xi_{0}.$
On $\widehat{\Delta}^{(n-1)}_{0}$   we define  the Poincar\'{e} map
$\pi_{n}=(f^{q_n}, f^{q_{n-1}}):\widehat{\Delta}^{(n-1)}_{0}\rightarrow \widehat{\Delta}^{(n-1)}_{0}$ as follows
\[\pi_{n}(\xi)=\left\{\begin{array}{ll}f^{q_n}(\xi), & \mbox{if \,\,
$\xi\in \Delta^{(n-1)}_{0}$},\\ f^{q_{n-1}}(\xi), & \mbox{if \,\,
$\xi\in \Delta^{(n)}_{0}$}.\end{array}\right.\]
Next we define the renormalization of $f$ as follows.
Let $\mathcal{A}_{n}:\mathbb{R}\rightarrow \mathbb{S}^{1}$ be an affine
covering map such that $\mathcal{A}_{n}([-1, 0])=\Delta_{0}^{(n-1)},$ with
$\mathcal{A}_{n}(0)=\xi_{0}$ and $\mathcal{A}_{n}(-1)=f^{q_{n-1}}(\xi_{0}).$
We define $a_{n}\in \mathbb{R}$ to be a positive number
such that $\mathcal{A}_{n}(a_{n})=f^{q_{n}}(\xi_{0}).$
 It is obvious that $\mathcal{A}_{n}:[0, a_{n}]\rightarrow \Delta^{(n)}_{0}$
and $\mathcal{A}_{n}:[-1, 0]\rightarrow \Delta^{(n-1)}_{0}.$
A pair of functions $(f_{n}, g_{n}):[-1, a_{n}]\rightarrow [-1, a_{n}]$
defined by $(f_{n}, g_{n})=\mathcal{A}^{-1}_{n}\circ \pi_{n}\circ \mathcal{A}_{n},$
is called the $n^{\text{th}}$ \emph{renormalization} of $f,$
where $\mathcal{A}^{-1}_{n}$ is the inverse branch that
maps $\widehat{\Delta}^{(n-1)}_{0}$  onto $[-1, a_{n}].$
Define the following M\"{o}bius transformation
$$
F_{n}:=F_{a_{n}, v_{n}, c_{n}}:z \rightarrow\frac{a_{n}+c_{n}z}{1-v_{n}z}
$$
where $c_{n}=c$ if $n$ is even, $c_{n}=c^{-1}$ if $n$ is odd, and
\begin{align*}
a_{n}&=\frac{|\Delta^{(n)}_{0}|}{|\Delta^{(n-1)}_{0}|},&
v_{n}&=\frac{c_{n}-a_{n}-b_{n}}{b_{n}},&
b_{n}&=\frac{|\Delta^{(n-1)}_{0}|-|\Delta^{(n)}_{q_{n-1}}|}{|\Delta^{(n-1)}_{0}|}.
\end{align*}
The following theorem has been proved in \cite{Renorm}.

\begin{thm}\label{coefbog}
Let $f \in \mathrm{D}^{1+\mathcal{Z}_{\gamma}}(\mathbb{S}^{1}\setminus \{\xi_{0}\})$ and $\gamma>1.$
Suppose the  rotation number of $f$ is irrational. There exists a constant $C=C(f)>0$ and a
 natural number $N_{0}=N_{0}(f)$  such that
%the following inequalities hold
\begin{align*}
\|f_{n}-F_{n}\|_{C^{1}([-1,0])}&\leq \frac{C}{n^{\gamma}},&
\|D^{2}f_{n}-D^{2}F_{n}\|_{C^{0}([-1,0])}&\leq \frac{C}{n^{\gamma-1}}
\end{align*}
for all $n\geq N_{0}.$
\end{thm}
The following lemma will be used in the subsequent  sections.

\begin{lem}\label{bounded}
Let $f \in \mathrm{D}^{1+\mathcal{Z}_{\gamma}}(\mathbb{S}^{1}\setminus \{\xi_{0}\})$ and $\gamma>1.$
Suppose the  rotation number of $f$ is irrational. There exists  a constant
$Q=Q(f)>0$   such that
\begin{align*}
\|f_{n}\|_{C^{2}([-1,0])}\leq Q.
\end{align*}
\end{lem}
\begin{proof}
The proof of the lemma implies from Theorem \ref{coefbog} and
Proposition 7.1 stated in \cite{Renorm}.
\end{proof}

\emph{Half-bounded rotation numbers.}
The half-bounded rotation numbers were defined by  Khanin and Teplinsky in \cite{KT2013} as follows.
Denote by $M_{o}$ and $M_{e}$ the class of all irrational rotation numbers
$\rho=[k_{1},k_{2},...),$ such that 
$$
M_{o}=\{\rho: \,(\exists C>0)\,(\forall m\in \mathbb{N})\,\, k_{2m-1}\leq C\},\,\,\,\,
M_{e}=\{\rho: \,(\exists C>0)\,(\forall m\in \mathbb{N})\,\, k_{2m}\leq C\}.
$$
Let us formulate the following theorem borrowed from \cite{KT2013}.
%%%%%%%%%%%%%%%%%%%%%%%%%%%
\begin{thm}\label{TK}
Let $f$ and $\tilde{f}$ be two $C^{2+\nu}$-smooth circle diffeomorphisms with breaks of the same size $c$
and the same rotation number $\rho\in M_{e}$ in case of $c>1,$ or $\rho\in M_{o}$ in case of
$0<c<1.$
There exist constants $C=C(f, \tilde{f})>0$ and $\mu\in (0,1)$ such that
$$
\|f_{n}-\tilde{f}_{n}\|_{C^{2}([-1,0])}\leq C\mu^{n}.
$$
\end{thm}
This theorem was extended by Khanin and Koci\'{c} \cite{KSE} for all irrational rotation numbers
and for the class of $\mathrm{D}^{1+\mathcal{Z}_{\gamma}}(\mathbb{S}^{1}\setminus \{\xi_{0}\})$
by Akhadkulov \textit{et al} \cite{Renorm}. More precisely, in  \cite{Renorm}, it was proven the following
%%%%%%%%%%%%%%%%%%%%%%%%%%%%%%%%%%
\begin{thm}\label{renormcon}
Let $f, \tilde{f}\in\mathrm{D}^{1+\mathcal{Z}_{\gamma}}(\mathbb{S}^{1}\setminus \{\xi_{0}\})$ and
$\gamma>1.$
Assume that $f$ and $\tilde{f}$ have the same break size $c$  and the same rotation number
$\rho\in M_{e}$ in the case of $c>1,$ or $\rho\in M_{o}$ in the case of $0<c<1.$
 There exists a constant  $C=C(f,\tilde{f})>0$ and a natural number $N_{0}=N_{0}(f,\tilde{f})$
such that
$$
\|f_{n}-\tilde{f}_{n}\|_{C^{1}([-1,0])}\leq \frac{C}{n^{\gamma}},
\,\,\,\,\,\,\,\,
\|D^{2}f_{n}-D^{2}\tilde{f}_{n}\|_{C^{0}([-1,0])}\leq \frac{C}{n^{\gamma-1}}
$$
for all $n\geq N_{0}.$
\end{thm}

\emph{An estimate of $Df_{n}.$}
The following set plays an important role in the investigations
of renormalizations of comuting pairs of M\"{o}bius transformations (see \cite{KT2013}).
$$
\Phi_{c}^{\varepsilon}=\{(a,v):\, \varepsilon<a<c-\varepsilon,\, \varepsilon<\frac{v}{c-1}<1-\varepsilon,
\,v+a-c+1>\varepsilon\},\,\, \varepsilon>0.
$$
\begin{lem}\label{A1B1}
Let $f\in\mathrm{D}^{1+\mathcal{Z}_{\gamma}}(\mathbb{S}^{1}\setminus \{\xi_{0}\}),$
$\gamma>1$ be a circle diffeomorphism with irrational rotation  $\rho$ and  the  break size $c.$
Assume that $\rho\in M_{e}$ if $c>1$ or $\rho\in M_{o}$ if $0<c<1.$
 There exists a constant $\varepsilon=\varepsilon(f)>0$ and a natural number $N_{0}=N_{0}(f)$
such that the projection $(a_{n}, v_{n})$ of the renormalization $(f_{n}, g_{n})$
belongs to $\Phi_{c_{n}}^{\varepsilon}$ for all $n\geq N_{0}.$
\end{lem}
\begin{proof}
The proof follows from Proposition 7.1 in  \cite{Renorm}.
\end{proof}

\begin{lem}\label{A1B2}
Let $f\in\mathrm{D}^{1+\mathcal{Z}_{\gamma}}(\mathbb{S}^{1}\setminus \{\xi_{0}\}),$
$\gamma>1$ be a circle diffeomorphism with irrational rotation  $\rho$ and  the  break size $c.$
Assume that $\rho\in M_{e}$ if $c>1$ or $\rho\in M_{o}$ if $0<c<1.$
 There exists a constant $\varepsilon=\varepsilon(f)>0$ and a natural number $N_{0}=N_{0}(f)$
such that, for all $n\geq N_{0},$ we have
$$
\frac{c_{n}}{(c_{n}+\varepsilon(1-c_{n}))^{2}}-\frac{C}{n^{\gamma}}\leq Df_{n}(z)\leq c^{2}_{n}-\varepsilon(c^{2}_{n}-1-\varepsilon(c_{n}-1))+\frac{C}{n^{\gamma}}
$$
if $c_{n}>1$ and
$$
c^{2}_{n}-\varepsilon(c^{2}_{n}-1-\varepsilon(c_{n}-1))-\frac{C}{n^{\gamma}}\leq Df_{n}(z) \leq \frac{c_{n}}{(c_{n}+(1-c_{n})\varepsilon)^{2}}+\frac{C}{n^{\gamma}}
$$
if $c_{n}<1.$
\end{lem}
\begin{proof}
It is easy to see that $DF_{n}(z)=(c_{n}+a_{n}v_{n})(1-v_{n}z)^{-2}.$ Let
$c_{n}>1.$ Lemma \ref{A1B1} implies $(c_{n}-1)\varepsilon< v_{n}<(c_{n}-1)(1-\varepsilon)$ and hence
$1+(c_{n}-1)\varepsilon< 1+v_{n}<c_{n}-\varepsilon(c_{n}-1).$ Using these inequalities
we get
\begin{multline}\label{20201}
DF_{n}(z)\leq c_{n}+a_{n}v_{n}<
c_{n}+(c_{n}-\varepsilon)(c_{n}-1)(1-\varepsilon)
=c^{2}_{n}-\varepsilon(c^{2}_{n}-1-\varepsilon(c_{n}-1))
\end{multline}
and
\begin{equation}\label{20201a}
DF_{n}(z)\geq\frac{c_{n}+a_{n}v_{n}}{(1+v_{n})^{2}}>
\frac{c_{n}+\varepsilon^{2}(c_{n}-1)}{(c_{n}+\varepsilon(1-c_{n}))^{2}}
>\frac{c_{n}}{(c_{n}-\varepsilon(c_{n}-1))^{2}}.
\end{equation}
Assume  $c_n<1.$ By Lemma \ref{A1B1} we have
$
(c_n-1)(1-\varepsilon)< v_n<(c_n-1)\varepsilon,
$
which implies that
$c_n+(1-c_n)\varepsilon<1+v_n<1+(c_n-1)\varepsilon$
and $(1-v_nz)^2>(1+v_n)^2.$ Hence we have
\begin{equation}\label{20202}
DF_n(z)\leq\frac{c_n+a_nv_n}{(1+v_n)^2}
<\frac{c_n-(1-c_n)\varepsilon^{2}}{(c_n+(1-c_n)\varepsilon)^2}
<\frac{c_n}{(c_n+(1-c_n)\varepsilon)^2}
\end{equation}
and
\begin{equation}\label{20202a}
DF_n(z)\geq c_n+a_nv_n>c_{n}+(c_{n}-\varepsilon)(c_{n}-1)(1-\varepsilon)
=c^{2}_{n}-\varepsilon(c^{2}_{n}-1-\varepsilon(c_{n}-1)).
\end{equation}
The proof of the lemma now follows from (\ref{20201})-(\ref{20202a}) and Theorem \ref{coefbog}.
\end{proof}
Denote $\mathfrak{c}=\max\{c, c^{-1}\}.$ It follows from Lemma \ref{A1B2}  the
following
\begin{cor}\label{corollary2020}
Let $f\in\mathrm{D}^{1+\mathcal{Z}_{\gamma}}(\mathbb{S}^{1}\setminus \{\xi_{0}\}),$
$\gamma>1$ be a circle diffeomorphism with irrational rotation  $\rho$ and  the  break size $c.$
Assume that $\rho\in M_{e}$ if $c>1$ or $\rho\in M_{o}$ if $0<c<1.$
 There exists a  natural number $N_{0}=N_{0}(f)$
such that
$$
\frac{1}{\mathfrak{c}^{2}}-\frac{C}{n^{\gamma}}\leq Df_{n}(z)\leq \mathfrak{c}^{2}+\frac{C}{n^{\gamma}}
$$
for all $n\geq N_{0}.$
\end{cor}

%%%%%%%%%%%%%%%%%%%%%%%%%%%%%%%%
%\marginpar{C \textcolor[rgb]{0.00,1.00,0.00}{haqida remark kerak}}
%%%%%%%%%%%%%%%%%%%%%%%%%%%%%%%
\section{Universal estimates for the segments  of $\mathcal{P}_{n}$}
In this section we estimate the ratio of lengths of segments of dynamical partition
of circle diffeomorphisms satisfying in the  setting of rotation number is bounded type.
\begin{lem}\label{comporision}
Let $f\in\mathrm{D}^{1+\mathcal{Z}_{\gamma}}(\mathbb{S}^{1}\setminus \{\xi_{0}\}),$
$\gamma>1$ be a circle diffeomorphism with the  break size $c$ and irrational rotation number $\rho$
of bounded type such that $s(\rho)=m.$ Let $\Delta^{(n+k)}\in \mathcal{P}_{n+k}$
such that $\Delta^{(n+k)}\subset \widehat{\Delta}_{0}^{(n-1)}$ where $k\geq 1.$
There exists a constant $C=C(f)>0$ and a natural number $N_{0}=N_{0}(f)$
such that
$$
\frac{|\Delta^{(n+k)}|}{|\widehat{\Delta}_{0}^{(n-1)}|}\leq C\lambda^{k}\Big(1+\frac{1}{n^{\gamma-1}}\Big)
$$
for all $n\geq N_{0},$ where $\lambda=\sqrt{\frac{\mathfrak{c}^{2}}{\mathfrak{c}^{2}+1}}.$
\end{lem}
\begin{proof}
First we show that
\begin{equation}\label{G0a}
 \frac{|\Delta^{(n+1)}_{0}|}{|\Delta^{(n-1)}_{0}|}\leq \lambda^{2}+\frac{C}{n^{\gamma}}
\end{equation}
for large enough $n.$ One can  verify that
$\Delta^{(n+1)}_{0}\subset \Delta^{(n)}_{k_{n+1}q_{n}+q_{n-1}}.$
 By (\ref{partition}) and Corollary \ref{corollary2020}
we have %\marginpar{Write Denjoy's ineq}
\begin{equation}\label{G0}
\begin{split}
 \frac{|\Delta^{(n+1)}_{0}|}{|\Delta^{(n-1)}_{0}|}  &\leq
 \frac{1}{1+\frac{|\Delta^{(n)}_{(k_{n+1}-1)q_{n}+q_{n-1}}|}{|\Delta^{(n+1)}_{0}|}}
 \leq \frac{1}{1+\frac{|\Delta^{(n)}_{(k_{n+1}-1)q_{n}+q_{n-1}}|}{|\Delta^{(n)}_{k_{n+1}q_{n}+q_{n-1}}|}}\\
 &\leq \frac{1}{1+(Df^{q_{n}}(\hat{\xi}))^{-1}}=\frac{1}{1+(Df_{n}(\hat{z}))^{-1}}
 \leq \frac{\mathfrak{c}^{2}}{\mathfrak{c}^{2}+1}+\frac{C}{n^{\gamma}}
 \end{split}
  \end{equation}
 where $\hat{\xi}\in \Delta^{(n)}_{(k_{n+1}-1)q_{n}+q_{n-1}}$ and $\hat{z}\in (-1,0)$
 such that $\mathcal{A}_{n}(\hat{z})=\hat{\xi}.$
 Inequality (\ref{G0}) yields
 \begin{equation}\label{G02020}
\begin{split}
 \frac{|\Delta^{(n+2l+1)}_{0}|}{|\Delta^{(n-1)}_{0}|}  &\leq
 \exp\Bigg(\sum_{s=0}^{l}\ln \Big(\lambda^{2}+\frac{C}{(n+2s)^{\gamma}}\Big)\Bigg)
 \leq \lambda^{2(l+1)}\Big(1+\frac{C}{n^{\gamma-1}}\Big).
 \end{split}
  \end{equation}
 Since the rotation number is bounded type
 we have
 \begin{equation}\label{G12020}
 \frac{|\Delta^{(n+k)}_{0}|}{|\widehat{\Delta}^{(n-1)}_{0}|}
 \leq C\lambda^{k}\Big(1+\frac{1}{n^{\gamma-1}}\Big)
  \end{equation}
  for any $k\geq 1$ and for $n$ large.
  Let $\Delta^{(n+k)}$ be any interval satisfying  $\Delta^{(n+k)}\in \mathcal{P}_{n+k}$
and $\Delta^{(n+k)}\subset \widehat{\Delta}_{0}^{(n-1)}$ where $k\geq 1.$
There exists $i_{0}$ such that $f^{i_{0}}(\Delta^{(n+k)}_{0})=\Delta^{(n+k)}.$
We claim that the length of intervals $\Delta_{0}^{(n+k)}$ and $f^{i_{0}}(\Delta^{(n+k)}_{0})$ are comparable,
that is, there exists a constant $C>1$ such that $C^{-1}\leq |\Delta_{0}^{(n+k)}|/|f^{i_{0}}(\Delta^{(n+k)}_{0})|\leq C.$
  Indeed, due to Finzi's inequality we have
   \begin{equation}\label{G22020}
 e^{-v}\leq\frac{|\Delta_{0}^{(n+k)}|}{|f^{i_{0}}(\Delta^{(n+k)}_{0})|}
 \frac{|f^{i_{0}}(\widehat{\Delta}^{(n-1)}_{0})|}{|\widehat{\Delta}^{(n-1)}_{0}|}
 \leq e^{v}.
  \end{equation}
  where $v$ is the total variation of $\log Df.$
  On the other hand the length of intervals $f^{i_{0}}(\widehat{\Delta}^{(n-1)}_{0})$
  and $\widehat{\Delta}^{(n-1)}_{0}$ are $(2e^v+1)$-comparable since
  $$
  f^{i_{0}}(\widehat{\Delta}^{(n-1)}_{0})\subset f^{-q_{n-1}}(\widehat{\Delta}^{(n-1)}_{0})\cup
  \widehat{\Delta}^{(n-1)}_{0}\cup f^{q_{n-1}}(\widehat{\Delta}^{(n-1)}_{0})
  $$ and
  $$
  \widehat{\Delta}^{(n-1)}_{0}\subset f^{-q_{n-1}+i_{0}}(\widehat{\Delta}^{(n-1)}_{0})\cup
  f^{i_{0}}(\widehat{\Delta}^{(n-1)}_{0})\cup f^{q_{n-1}+i_{0}}(\widehat{\Delta}^{(n-1)}_{0}).
  $$
Therefore the length of intervals $\Delta_{0}^{(n+k)}$ and $f^{i_{0}}(\Delta^{(n+k)}_{0})$ are comparable.
This and inequality (\ref{G12020}) imply
$$
\frac{|\Delta^{(n+k)}|}{|\widehat{\Delta}_{0}^{(n-1)}|}\leq C\lambda^{k}\Big(1+\frac{1}{n^{\gamma-1}}\Big)
$$
for $k\geq 1$ and large enough  $n.$
 \end{proof}

%%%%%%%%%%%%%%%%%%%%%%%%%%%%%%%%%%%%%
\section{Closeness of rescaled points}
Our aim in this section is to show the closeness of rescaled points
of $\xi$ and $h(\xi).$
Let $f$  be a circle diffeomorphism  with a break.
Let $\mathcal{A}_{n}$
be the affine covering map of $f.$
Denote by $\mathfrak{r}_{n}:\widehat{\Delta}_{0}^{(n-1)}\rightarrow [-1, a_{n}]$
 the inverse of $\mathcal{A}_{n}.$
The point  $\mathfrak{r}_{n}(\xi)$ is called rescaled point of $\xi.$
 Next consider two circle diffeomorphisms $f$ and  $\tilde{f}$ with a break and with the identical irrational rotation number.
 Define  the distance between appropriately rescaled points of $\xi$
and $h(\xi):$
$$
\mathfrak{d}_{n}(\xi)=|\mathfrak{r}_{n}(\xi)-\mathfrak{\tilde{r}}_{n}(h(\xi))|.
$$
We have
%%%%%%%%%%%%%%%%%%%%%%%%%%%%%%%%%%%%%%%%%%%%%%%%%%%%%%%%
\begin{lem}\label{Closeness}
Let $f$ and $\tilde{f}$ satisfy the assumptions  of Theorem \ref{Main}.
 Then for any $\alpha\in (0, \gamma)$
there exist   $\kappa=\kappa(f,\tilde{f})>1,$ $C=C(f,\tilde{f})>0$
and $N_{0}=N_{0}(f, \tilde{f})\in \mathbb{N}$
such that
$$
\mathfrak{d}_{n}(\xi)\leq \frac{C}{n^{\gamma-\alpha}}
$$
for all  $\xi \in \Xi^{*}_{\ell}\cap \widehat{\Delta}_{0}^{(n-1)}$
provided $n\leq \ell \leq n+[\alpha \log_{\kappa}n]$ for $n\geq N_{0}$ where $[\cdot]$ is the integer part of a number.
\end{lem}
%%%%%%%%%%%%%%%%%%%%%%%%%%%%%%%%%%%%%%%%%%%%%%%%%%%%%
\begin{proof}
 It is easy to verify that
$\Xi^{*}_{\ell}\cap \widehat{\Delta}_{0}^{(n-1)}=\{\xi_{q_{n-1}},
\xi_{q_{n}+q_{n-1}}, \xi_{0}, \xi_{q_{n}}\}$ for $\ell=n.$ One can easily see that
$\mathfrak{d}_{n}(\xi_{q_{n-1}})=\mathfrak{d}_{n}(\xi_{0})=0,$
$\mathfrak{d}_{n}(\xi_{q_{n}+q_{n-1}})=|f_{n}(-1)-\tilde{f}_{n}(-1)|$
and $\mathfrak{d}_{n}(\xi_{q_{n}})=|f_{n}(0)-\tilde{f}_{n}(0)|.$
Hence by Theorem \ref{renormcon} we get
\begin{equation}\label{H1}
\underset{\xi\in \Xi^{*}_{n}\cap \widehat{\Delta}_{0}^{(n-1)}}{\max}\mathfrak{d}_{n}(\xi)\leq \frac{C}{n^{\gamma}}
\end{equation}
for large enough $n.$
For fixed $\ell>n$ let us denote
$\mathfrak{q}_{n}=\max_{\xi\in \Xi^{*}_{\ell}\cap \widehat{\Delta}_{0}^{(n-1)}}\mathfrak{d}_{n}(\xi).$
The obvious equality
$
\mathfrak{d}_{n}(\xi)=|f_{n}(0)\mathfrak{r}_{n+1}(\xi)-\tilde{f}_{n}(0)\mathfrak{\tilde{r}}_{n+1}(h(\xi))|
$
and Theorem \ref{renormcon} imply
\begin{equation}\label{H2}
\mathfrak{d}_{n}(\xi)\leq a_{n}\mathfrak{d}_{n+1}(\xi)+
\frac{C}{n^{\gamma}}
\end{equation}
 if $\xi\in \Xi^{*}_{\ell}\cap \widehat{\Delta}_{0}^{(n)}$ and $n$ is large,
 where $a_{n}=f_{n}(0)=|\Delta^{(n)}_{0}|/|\Delta^{(n-1)}_{0}|.$
 Let $\xi\in \Xi_{\ell}\cap \check{\Delta}^{(n-1)}_{0}.$
 Consider an arbitrary thread in the decomposition (\ref{J1}) and denote
 $\eta_{s}=\mathfrak{r}_{n}(\xi_{l+sq_{n}+q_{n-1}}),$
 $\tilde{\eta}_{s}=\mathfrak{\tilde{r}}_{n}(\tilde{\xi}_{l+sq_{n}+q_{n-1}}),$
 for $0\leq s \leq k_{n+1},$ so that
 $\mathfrak{d}_{n}(\xi_{l+sq_{n}+q_{n-1}})=|\eta_{s}-\tilde{\eta}_{s}|$
  where $\tilde{\xi}_{l+sq_{n}+q_{n-1}}=h(\xi_{l+sq_{n}+q_{n-1}}).$
  It is easy to see that  $\eta_{s+1}=f_{n}(\eta_{s})$ and
  $\tilde{\eta}_{s+1}=\tilde{f}_{n}(\tilde{\eta}_{s}).$ First we consider the case $s=0.$
  In this case, it is a simple matter to verify that
\begin{equation}\label{H3}
\begin{split}
\mathfrak{d}_{n}(\xi_{l+q_{n-1}})& =|\eta_{0}-\tilde{\eta}_{0}|=\Big|\frac{\mathfrak{r}_{n-1}(\xi_{l+q_{n-1}})}{f_{n-1}(0)}-
  \frac{\mathfrak{\tilde{r}}_{n-1}(\tilde{\xi}_{l+q_{n-1}})}{\tilde{f}_{n-1}(0)}\Big|\\
 &\leq \frac{\mathfrak{d}_{n-1}(\xi_{l+q_{n-1}})}{f_{n-1}(0)}+
 \Big|\frac{1}{f_{n-1}(0)}-\frac{1}{\tilde{f}_{n-1}(0)}\Big|\Big|\mathfrak{\tilde{r}}_{n-1}(\tilde{\xi}_{l+q_{n-1}})\Big|,\\
 \mathfrak{d}_{n-1}(\xi_{l+q_{n-1}})& =|f_{n-1}(\mathfrak{r}_{n-1}(\xi_{l}))-
  \tilde{f}_{n-1}(\mathfrak{\tilde{r}}_{n-1}(\tilde{\xi}_{l}))|\\
 %&\leq |\tilde{f}_{n-1}(\mathfrak{r}_{n-1}(\xi_{l}))-\tilde{f}_{n-1}(\mathfrak{\tilde{r}}_{n-1}(\tilde{\xi}_{l}))|\\
% &+|f_{n-1}(\mathfrak{r}_{n-1}(\xi_{l}))-
%  \tilde{f}_{n-1}(\mathfrak{r}_{n-1}(\xi_{l}))|\\
  &\leq Df_{n-1}(\mathfrak{r}^{0})\mathfrak{d}_{n-1}(\xi_{l})+
  |f_{n-1}(\mathfrak{\tilde{r}}_{n-1}(\xi_{l}))-
  \tilde{f}_{n-1}(\mathfrak{\tilde{r}}_{n-1}(\xi_{l}))|,\\
  \mathfrak{d}_{n-1}(\xi_{l})&=
  |f_{n-1}(0)f_{n}(0)\mathfrak{r}_{n+1}(\xi_{l})-\tilde{f}_{n-1}(0)\tilde{f}_{n}(0)\mathfrak{\tilde{r}}_{n+1}(\tilde{\xi}_{l})|\\
  &\leq f_{n-1}(0)f_{n}(0) \mathfrak{d}_{n+1}(\xi_{l})+
  |f_{n-1}(0)f_{n}(0)-\tilde{f}_{n-1}(0)\tilde{f}_{n}(0)||\mathfrak{\tilde{r}}_{n+1}(\tilde{\xi}_{l})|,
  \end{split}
\end{equation}
where $\mathfrak{r}^{0}$ is a point between  $\mathfrak{r}_{n-1}(\xi_{l})$ and $\mathfrak{\tilde{r}}_{n-1}(\tilde{\xi}_{l})$
such that $|f_{n-1}(\mathfrak{r}_{n-1}(\xi_{l}))-
  f_{n-1}(\mathfrak{\tilde{r}}_{n-1}(\tilde{\xi}_{l}))|=Df_{n-1}(\mathfrak{r}^{0})|\mathfrak{r}_{n-1}(\xi_{l})-
  \mathfrak{\tilde{r}}_{n-1}(\tilde{\xi}_{l})|.$
Since the rotation number is bounded type,
  Theorem \ref{renormcon} and inequalities  (\ref{G0a}) and (\ref{H3})
  imply that
\begin{equation}\label{H4}
\begin{split}
\mathfrak{d}_{n}(\xi_{l+q_{n-1}}) =|\eta_{0}-\tilde{\eta}_{0}|&\leq a_{n}
Df_{n-1}(\mathfrak{r}^{0})\mathfrak{d}_{n+1}(\xi_{l})+\frac{C}{n^{\gamma}}
%&\leq a_{n}e^{v_{f}}\mathfrak{d}_{n+1}(\xi_{l})+\frac{C}{n^{\gamma}}
\end{split}
\end{equation}
for $n$ large.
Now consider the case $0<s<k_{n+1}.$ Let $\mathfrak{r}^{s}$ be a point between
$\eta_{s-1}$ and $\tilde{\eta}_{s-1}$ such that
$|f_{n}(\eta_{s-1})-f_{n}(\tilde{\eta}_{s-1})|=Df_{n}(\mathfrak{r}^{s})|\eta_{s-1}- \tilde{\eta}_{s-1}|.$
Then we have
\begin{equation*}
\begin{split}
\mathfrak{d}_{n}(\xi_{l+sq_{n}+q_{n-1}}) =|\eta_{s}-\tilde{\eta}_{s}|&\leq
Df_{n}(\mathfrak{r}^{s})|\eta_{s-1}- \tilde{\eta}_{s-1}|+\frac{C}{n^{\gamma}}
%&\leq e^{v_{f}}|\eta_{s-1}- \tilde{\eta}_{s-1}|+\frac{C}{n^{\gamma}}
\end{split}
\end{equation*}
for $n$ large.
Iterating into it we get
\begin{equation}\label{H5}
\mathfrak{d}_{n}(\xi_{l+sq_{n}+q_{n-1}}) =|\eta_{s}-\tilde{\eta}_{s}|\leq
\prod_{i=1}^{s}Df_{n}(\mathfrak{r}^{i})|\eta_{0}-\tilde{\eta}_{0}|+\Big(1+\sum_{j=2}^{s}\prod_{i=j}^{s}Df_{n}(\mathfrak{r}^{i})\Big)\frac{C}{n^{\gamma}}
\end{equation}
Since the rotation number is bounded type the expressions $\Big(1+\sum_{j=2}^{s}\prod_{i=j}^{s}Df_{n}(\mathfrak{r}^{i})\Big)$ and $\prod_{i=1}^{s}Df_{n}(\mathfrak{r}^{i})$
are bounded above by a universal constant.
This and  relations  (\ref{H4}) and (\ref{H5}) imply
\begin{equation}\label{H50}
\mathfrak{d}_{n}(\xi_{l+sq_{n}+q_{n-1}})\leq \prod_{i=1}^{s}Df_{n}(\mathfrak{r}^{i})\Big( a_{n} Df_{n-1}(\mathfrak{r}^{0})\Big)  \mathfrak{d}_{n+1}(\xi_{l})+\frac{C}{n^{\gamma}}
\end{equation}
for $n$ large.
Finally, consider the case $s=k_{n+1}.$
In this case, it is easy to see  that
\begin{equation}\label{H6}
\begin{split}
\mathfrak{d}_{n}(\xi_{l+q_{n+1}})& =\Big|\mathfrak{r}_{n}(\xi_{l+q_{n+1}})-
  \mathfrak{\tilde{r}}_{n}(\tilde{\xi}_{l+q_{n+1}})\Big|\\
  &=|f_{n}(0)f_{n+1}(\mathfrak{r}_{n+1}(\xi_{l}))-
  \tilde{f}_{n}(0)\tilde{f}_{n+1}(\mathfrak{\tilde{r}}_{n+1}(\xi_{l}))|\\
  &\leq a_{n} Df_{n+1}(\mathfrak{r}^{k_{n+1}})\mathfrak{d}_{n+1}(\xi_{l})+\frac{C}{n^{\gamma}}
\end{split}
\end{equation}
for large enough $n,$ where $\mathfrak{r}^{k_{n+1}}$ is a point between  $\mathfrak{r}_{n+1}(\xi_{l})$ and $\mathfrak{\tilde{r}}_{n+1}(\tilde{\xi}_{l})$
such that $|f_{n+1}(\mathfrak{r}_{n+1}(\xi_{l}))-
  f_{n+1}(\mathfrak{\tilde{r}}_{n+1}(\tilde{\xi}_{l}))|=Df_{n+1}(\mathfrak{r}^{k_{n+1}})|\mathfrak{r}_{n+1}(\xi_{l})-
  \mathfrak{\tilde{r}}_{n+1}(\tilde{\xi}_{l})|.$
 Combining Lemmas \ref{A1B1} and  \ref{A1B2}
we can easily obtain  that $a_{n}Df_{n-1}(\mathfrak{r}^{0}),$ $a_{n}Df_{n+1}(\mathfrak{r}^{k_{n+1}})\leq c_{n}^{2}+Cn^{-\gamma}$
if $c_{n}>1$ and  $a_{n}Df_{n-1}(\mathfrak{r}^{0}),$ $a_{n}Df_{n+1}(\mathfrak{r}^{k_{n+1}})\leq c_{n}^{-1}+Cn^{-\gamma}$
if $c_{n}<1$ and
\begin{equation}\label{20203}
 \prod_{i=1}^{s}Df_{n}(\mathfrak{r}^{i})\Big( a_{n} Df_{n-1}(\mathfrak{r}^{0})\Big)
 \leq\begin{cases}
               c_{n}^{2(s+1)}+Cn^{-\gamma}, & \mbox{if}\,\,\, c_n>1\\

               c_{n}^{-(s+1)}+Cn^{-\gamma},& \mbox{if}\,\,\, c_n<1
             \end{cases}
\end{equation}
for $n$ large.
Let us denote $\mathfrak{c}=\max\{c, c^{-1}\}$ and $\kappa:=\kappa(c, m)=\mathfrak{c}^{2m}.$
 It follows from the relations (\ref{H2}), (\ref{H4}), (\ref{H50}),(\ref{H6}) and (\ref{20203}) that
\begin{equation}\label{H7}
\mathfrak{q}_{n}\leq \kappa \mathfrak{q}_{n+1}+\frac{C}{n^{\gamma}}
\end{equation}
for $n$ large.
Iterating (\ref{H7}) we get
$$
\mathfrak{q}_{n}\leq \kappa^{\ell-n}\mathfrak{q}_{\ell}+C\sum_{j=n}^{\ell-1}\frac{\kappa^{j-n}}{j^{\gamma}}
$$
for $n$ large.
Inequality (\ref{H1}) implies $\mathfrak{q}_{\ell}\leq C \ell^{-\gamma}.$
Hence
\begin{equation}\label{H8}
\mathfrak{q}_{n}\leq C\sum_{j=n}^{\ell}\frac{\kappa^{j-n}}{j^{\gamma}}\leq \frac{C\kappa^{\ell-n}}{n^{\gamma}}.
\end{equation}
The condition $n\leq \ell \leq n+[\alpha\log_{\kappa}n]$ makes it obvious that
$$
\mathfrak{q}_{n}\leq \frac{C}{n^{\gamma-\alpha}}
$$
for large enough $n.$
Lemma \ref{Closeness} is proved.
 \end{proof}
\section{Proof of main theorem}
In this section we prove our main theorem. For this,
first we prove a preparatory lemma and then we prove $C^{1}$-smoothness
of the conjugacy. Finally, we prove  $C^{1+\omega_{\gamma}}$-smoothness
of the conjugacy.
\subsection{Preparatory lemma}
We begin by proving  the following lemma.
\begin{lem}\label{Lambda}
Let $f$ and $\tilde{f}$ satisfy the assumptions of Theorem \ref{Main}.
 Then there exists a constant $C:=C(f, \tilde{f})>0$ and a natural number $N_{0}:=N_{0}(f, \tilde{f})$
 such that
 $$
 \Lambda_{n}\leq \frac{C}{n^{\frac{\gamma}{2}}}
 $$
 for all $n\geq N_{0}.$
\end{lem}
\begin{proof}
One can see that
\begin{equation}\label{L0}
\begin{split}
|\Lambda_{n}(\xi)|&=|\log Df^{q_{n}}(\xi)-\log D\tilde{f}^{q_{n}}(h(\xi))|=|\log Df_{n}(\mathfrak{r}_{n}(\xi))-\log D\tilde{f}_{n}(\tilde{\mathfrak{r}}_{n}(\tilde{\xi}))|\\
&\leq |\log Df_{n}(\mathfrak{r}_{n}(\xi))-\log Df_{n}(\tilde{\mathfrak{r}}_{n}(\tilde{\xi}))|+
|\log Df_{n}(\tilde{\mathfrak{r}}_{n}(\tilde{\xi}))-\log D\tilde{f}_{n}(\tilde{\mathfrak{r}}_{n}(\tilde{\xi}))|\\
&\leq \|D\log Df_{n}\|_{C^{0}([-1,0])}\mathfrak{d}_{n}(\xi)+\frac{1}{\inf D\tilde{f}_{n}}
\|Df_{n}-D\tilde{f}_{n}\|_{C^{0}([-1,0])}.
\end{split}
\end{equation}
By Lemma \ref{bounded}
we have $\|D\log Df_{n}\|_{C^{0}([-1,0])}\leq Q.$
Denjoy's inequality implies $(\inf D\tilde{f}_{n})^{-1}\leq e^{v_{f}}.$
From Theorem \ref{renormcon}
it follows that $\|Df_{n}-D\tilde{f}_{n}\|_{C^{0}([-1,0])}\leq C n^{-\gamma}$
for $n$ large.
Next we estimate $\mathfrak{d}_{n}(\xi)$ on $\widehat{\Delta}_{0}^{(n-1)}.$
First we assume that $\xi\in\Xi^{*}_{n+[\frac{\gamma}{2}\log_{\kappa}n]}\cap \widehat{\Delta}_{0}^{(n-1)}.$
Then, if we choose $\alpha=\gamma/2$ in  Lemma \ref{Closeness}
then for large enough $n,$ the function $\mathfrak{d}_{n}(\xi)$ can be estimated as follows
\begin{equation}\label{L1}
\mathfrak{d}_{n}(\xi)\leq \frac{C}{n^{\frac{\gamma}{2}}}.
\end{equation}
Let $\xi$ be any point of $\widehat{\Delta}_{0}^{(n-1)}.$
Denote by $\Delta^{(n+[\frac{\gamma}{2}\log_{\kappa}n])}(\xi)$ the segment  of
$\mathcal{P}_{n+[\frac{\gamma}{2}\log_{\kappa}n]}$ containing the point $\xi$
and $r_{n}(\xi):=r_{n+[\frac{\gamma}{2}\log_{\kappa}n]}(\xi)$ the right endpoint of
$\Delta^{(n+[\frac{\gamma}{2}\log_{\kappa}n])}(\xi).$
A trivial reasoning shows that
\begin{equation}\label{L2}
\begin{split}
\mathfrak{d}_{n}(\xi)&=|\mathfrak{r}_{n}(\xi)-\mathfrak{\tilde{r}}_{n}(h(\xi))|\\
&\leq \Big|\frac{\xi-r_{n}(\xi)}{|\Delta^{(n-1)}_{0}|} -\frac{h(\xi)-h(r_{n}(\xi))}
{|\tilde{\Delta}^{(n-1)}_{0}|}\Big|+\mathfrak{d}_{n}(r_{n}(\xi))\\
&\leq \frac{|\Delta^{(n+[\frac{\gamma}{2}\log_{\kappa}n])}(\xi)|}{|\Delta^{(n-1)}_{0}|} +
\frac{|\tilde{\Delta}^{(n+[\frac{\gamma}{2}\log_{\kappa}n])}(h(\xi))|}{|\tilde{\Delta}^{(n-1)}_{0}|}+
\mathfrak{d}_{n}(r_{n}(\xi))
\end{split}
\end{equation}
where $\tilde{\Delta}^{(n+[\frac{\gamma}{2}\log_{\kappa}n])}(h(\xi))$ the segment  of
$\mathcal{\tilde{P}}_{n+[\frac{\gamma}{2}\log_{\kappa}n]}:=\mathcal{\tilde{P}}_{n+[\frac{\gamma}{2}\log_{\kappa}n]}(h(\xi_{0}), \tilde{f})$
containing the point $h(\xi).$
By (\ref{L1}) we have
\begin{equation}\label{L3}
\mathfrak{d}_{n}(r_{n}(\xi))\leq \frac{C}{n^{\frac{\gamma}{2}}}.
\end{equation}
It follows easily from Lemma  \ref{comporision}  that
\begin{equation}\label{L4}
\frac{|\Delta^{(n+[\frac{\gamma}{2}\log_{\kappa}n])}(\xi)|}{|\Delta^{(n-1)}_{0}|}\leq C
\lambda^{\frac{\gamma}{2}\log_{\kappa}n}\Big(1+\frac{1}{n^{\gamma-1}}\Big),
\end{equation}
and
\begin{equation}\label{L5}
\frac{|\tilde{\Delta}^{(n+[\frac{\gamma}{2}\log_{\kappa}n])}(h(\xi))|}{|\tilde{\Delta}^{(n-1)}_{0}|}\leq
C\lambda^{\frac{\gamma}{2}\log_{\kappa}n}\Big(1+\frac{1}{n^{\gamma-1}}\Big).
\end{equation}
One can see that
 \begin{equation}\label{L6}
\lambda^{\frac{\gamma}{2}\log_{\kappa}n}
=\Big(\frac{1}{n^{\frac{\gamma}{2}}}\Big)^{\log_{\kappa}\frac{1}{\lambda}}
 \end{equation}
Hypothesis $(d)$ of Theorem \ref{Main}
implies that $\lambda^{-1}> \kappa.$ Hence
 $$
 \log_{\kappa}\frac{1}{\lambda}>1.
 $$
This implies
\begin{equation}\label{L7}
\lambda^{\frac{\gamma}{2}\log_{\kappa}n}
\leq\frac{1}{n^{\frac{\gamma}{2}}}
 \end{equation}
Combining (\ref{L0})-(\ref{L7}) we conclude that
$$
 \Lambda_{n}\leq \frac{C}{n^{\frac{\gamma}{2}}}
 $$
 for large enough $n.$ Lemma \ref{Lambda} is proved.
\end{proof}
\subsection{$C^{1}$-smoothness of conjugacy}
By the hypotheses of Theorem \ref{Main}  the rotation number of $f$ and $\tilde{f}$ is bounded type
and $\gamma>2.$ Lemma \ref{Lambda}  implies that
$$
\sum_{n=0}^{\infty}k_{n+1}\Lambda_{n}<\infty.
$$
Therefore, it follows from Theorem \ref{criteria} that
 the cohomological equation (\ref{F0a}) has a continuous solution
 $\zeta.$
Next we prove the following lemma.
\begin{lem}\label{c1 smoothness}
There exists $\beta>0$ such that
$$
Dh(\xi)=\beta e^{\zeta(\xi)},\,\,\,\,\,\,\,\,\,\,\,
\text{for all}\,\,\,\,\,\,\,\,\,\, \xi\in \mathbb{S}^{1}.
$$
\end{lem}
\begin{proof}
Denote by
$\beta_{n}=|\tilde{\Delta}^{(n)}_{0}|/|\Delta^{(n)}_{0}|.$
Since the rotation number of $f$ and $\tilde{f}$ is bounded type,
Theorem \ref{renormcon} and inequality (\ref{G0a}) imply that
\begin{equation}\label{L7}
|\ln \beta_{n}-\ln \beta_{n-1}|=|\ln f_{n}(0)-\ln \tilde{f}_{n}(0)|
\leq \frac{1}{L_{m}(\hat{v})}|f_{n}(0)-\tilde{f}_{n}(0)|\leq \frac{C}{n^{\gamma
}}.
\end{equation}
Since $\gamma>2$ the sequence  $(\ln \beta_{n})_{n}$ and as well as $(\beta_{n})_{n}$ is
convergent. Let $\beta=\lim_{n \rightarrow \infty}\beta_{n}.$
It follows from
$$
\lim_{n\rightarrow \infty}\frac{|\tilde{\Delta}^{(n)}_{0}|}{|\Delta^{(n)}_{0}|}
=\lim_{n\rightarrow \infty}\frac{|h(\Delta^{(n)}_{0})|}{|\Delta^{(n)}_{0}|}=Dh(\xi_{0})
$$
and $\zeta(\xi_{0})=0$ that $Dh(\xi_{0})=\beta e^{\zeta(\xi_{0})}.$
From this and the equality $h\circ f=\tilde{f}\circ h$
we deduce
\begin{equation}\label{L7a1}
\log Dh(\xi_{i})-\log Dh(\xi_{i-1})=\log D\tilde{f}(h(\xi_{i-1}))-\log Df(\xi_{i-1}).
\end{equation}
for any $\xi_{i} \in \Xi$ where $\xi_{i}=f^{i}(\xi_{0}),$ $i\geq 1.$
 The cohomological equation (\ref{F0a})
implies that
\begin{equation}\label{L7a2}
\zeta(\xi_{i})-\zeta(\xi_{i-1})=\log D\tilde{f}(h(\xi_{i-1}))-\log Df(\xi_{i-1}).
\end{equation}
Combining (\ref{L7a1}) and (\ref{L7a2}) we get
$$
\log Dh(\xi_{i})-\zeta(\xi_{i})=\log Dh(\xi_{i-1})-\zeta(\xi_{i-1})
$$
which implies
 $$
 \log Dh(\xi_{i})-\zeta(\xi_{i})=\log Dh(\xi_{0})-\zeta(\xi_{0}).
$$
Hence
\begin{equation}\label{L8}
Dh(\xi_{i})=\beta e^{\zeta(\xi_{i})}
\end{equation}
for any $\xi_{i} \in \Xi.$
Since $\zeta$ is continuous and $\Xi$ is dense in $\mathbb{S}^{1}$   the function
$Dh$ can be continuously extended to the whole of $\mathbb{S}^{1}$
verifying the equality (\ref{L8}). This proves Lemma \ref{c1 smoothness}
and concludes the  $C^{1}$-smoothness of the conjugacy.
\end{proof}

\subsection{$C^{1+\omega_{\gamma}}$-smoothness of conjugacy}
It follow from $C^{1}$-smoothness of conjugacy
and the equality $h\circ f=\tilde{f}\circ h$
that
\begin{equation}\label{T1}
\log Dh\circ f-\log Dh=\log D\tilde{f}\circ h- \log Df.
\end{equation}
Consider the points $\xi_{i}$ and $\xi_{i+q_{n-1}+sq_{n}}$ where $1\leq s \leq k_{n+1}.$
It is clear that $\xi_{i}, \xi_{i+q_{n-1}+sq_{n}}\in \Delta^{(n-1)}_{i}.$
The relation (\ref{T1}) implies
$$
 |\log Dh(\xi_{i+q_{n-1}+sq_{n}})-\log Dh(\xi_{i})|\leq s\Lambda_{n}+\Lambda_{n-1}.
$$
Consequently, for any $\xi_{j}\in \Xi\cap \check{\Delta}^{(n-1)}_{i}$ we have
$$
 |\log Dh(\xi_{j})-\log Dh(\xi_{i})|\leq C\sum_{\ell=n}^{\infty}k_{\ell+1}\Lambda_{\ell}.
$$
Since $k_{\ell+1}$ is bounded from Lemma \ref{Lambda} it implies that
\begin{equation}\label{T2}
 |\log Dh(\xi_{j})-\log Dh(\xi_{i})|\leq C\sum_{\ell=n}^{\infty}\frac{1}{\ell^{\frac{\gamma}{2}}}\leq \frac{C}{n^{\frac{\gamma}{2}-1}}.
\end{equation}
It is obvious that
\begin{equation}\label{T2a1}
|\Delta^{(n+1)}_{i}|\leq |\xi_{j}-\xi_{i}|\leq |\Delta^{(n-1)}_{i}|.
\end{equation}
Lemma  \ref{comporision} implies that there exist $\mu_{1},\mu_{2}\in(0,1)$
verifying $\mu_{1}<\mu_{2}$ such that
\begin{equation}\label{T2a2}
\mu_{1}^{n}\leq |\Delta^{(n)}|\leq \mu_{2}^{n}
\end{equation}
for any $\Delta^{(n)}\in \mathcal{P}_{n}.$ Relations (\ref{T2a1}) and (\ref{T2a2})
imply
\begin{equation}\label{T3}
n=\mathcal{O}\Big(\frac{1}{\big|\log|\xi_{j}-\xi_{i}|\big|}\Big).
\end{equation}
 Combining (\ref{T2}) with (\ref{T3})
 we can assert that
  \begin{equation}\label{T4}
 |\log Dh(\xi_{j})-\log Dh(\xi_{i})|\leq \frac{C}{\big|\log|\xi_{j}-\xi_{i}|\big|^{\frac{\gamma}{2}-1}}.
\end{equation}
Since $\Xi$ is dense in $\mathbb{S}^{1},$ the function $Dh$ can be continuously extended to the whole of $\mathbb{S}^{1}$ verifying the inequality (\ref{T4}).
This proves $C^{1+\omega_{\gamma}}$-smoothness of the conjugacy.
Theorem \ref{Main} is proved.

%\section

\end{document}